\documentclass[12pt]{amsart}
\usepackage{amsmath,amssymb,amsfonts,amsthm}
\usepackage[all,cmtip]{xy}
\usepackage{fullpage}
\usepackage[francais,english]{babel}
\begin{document}
\theoremstyle{plain}
\newtheorem{thm}{Theorem}[section]
\newtheorem{lem}[thm]{Lemma}
\newtheorem{cor}[thm]{Corollary}
\newtheorem{prop}[thm]{Proposition}
\newtheorem{rem}[thm]{Remark}
\newtheorem{defn}[thm]{Definition}
\newtheorem{ex}[thm]{Example}
\newtheorem{ques}[thm]{Question}
\newtheorem{fact}[thm]{Fact}
\newtheorem{conj}[thm]{Conjecture}
\numberwithin{equation}{subsection}
\def\theequation{\thesection.\arabic{equation}}
\newcommand{\mc}{\mathcal}
\newcommand{\mb}{\mathbb}
\newcommand{\surj}{\twoheadrightarrow}
\newcommand{\inj}{\hookrightarrow}
\newcommand{\zar}{{\rm zar}}
\newcommand{\an}{{\rm an}} 
\newcommand{\red}{{\rm red}}
\newcommand{\codim}{{\rm codim}}
\newcommand{\rank}{{\rm rank}}
\newcommand{\Ker}{{\rm Ker \,}}
\newcommand{\Pic}{{\rm Pic}}
\newcommand{\Div}{{\rm Div}}
\newcommand{\Hom}{{\rm Hom}}
\newcommand{\im}{{\rm im \,}}
\newcommand{\Spec}{{\rm Spec \,}}
\newcommand{\Sing}{{\rm Sing}}
\newcommand{\Char}{{\rm char}}
\newcommand{\Tr}{{\rm Tr}}
\newcommand{\Gal}{{\rm Gal}}
\newcommand{\Min}{{\rm Min \ }}
\newcommand{\Max}{{\rm Max \ }}
\newcommand{\CH}{{\rm CH}}
\newcommand{\pr}{{\rm pr}}
\newcommand{\cl}{{\rm cl}}
\newcommand{\gr}{{\rm Gr }}
\newcommand{\Coker}{{\rm Coker \,}}
\newcommand{\id}{{\rm id}}
\newcommand{\Rep}{{\bold {Rep} \,}}
\newcommand{\Aut}{{\rm Aut}}
\newcommand{\GL}{{\rm GL}}
\newcommand{\Bl}{{\rm Bl}}
\newcommand{\Jab}{{\rm Jab}}
\newcommand{\alb}{\rm Alb}
\newcommand{\NS}{\rm NS}
\newcommand{\sA}{{\mathcal A}}
\newcommand{\sB}{{\mathcal B}}
\newcommand{\sC}{{\mathcal C}}
\newcommand{\sD}{{\mathcal D}}
\newcommand{\sE}{{\mathcal E}}
\newcommand{\sF}{{\mathcal F}}
\newcommand{\sG}{{\mathcal G}}
\newcommand{\sH}{{\mathcal H}}
\newcommand{\sI}{{\mathcal I}}
\newcommand{\sJ}{{\mathcal J}}
\newcommand{\sK}{{\mathcal K}}
\newcommand{\sL}{{\mathcal L}}
\newcommand{\sM}{{\mathcal M}}
\newcommand{\sN}{{\mathcal N}}
\newcommand{\sO}{{\mathcal O}}
\newcommand{\sP}{{\mathcal P}}
\newcommand{\sQ}{{\mathcal Q}}
\newcommand{\sR}{{\mathcal R}}
\newcommand{\sS}{{\mathcal S}}
\newcommand{\sT}{{\mathcal T}}
\newcommand{\sU}{{\mathcal U}}
\newcommand{\sV}{{\mathcal V}}
\newcommand{\sW}{{\mathcal W}}
\newcommand{\sX}{{\mathcal X}}
\newcommand{\sY}{{\mathcal Y}}
\newcommand{\sZ}{{\mathcal Z}}
\newcommand{\A}{{\mathbb A}}
\newcommand{\B}{{\mathbb B}}
\newcommand{\C}{{\mathbb C}}
\newcommand{\D}{{\mathbb D}}
\newcommand{\E}{{\mathbb E}}
\newcommand{\F}{{\mathbb F}}
\newcommand{\G}{{\mathbb G}}
\renewcommand{\H}{{\mathbb H}}
\newcommand{\I}{{\mathbb I}}
\newcommand{\J}{{\mathbb J}}
\newcommand{\M}{{\mathbb M}}
\newcommand{\N}{{\mathbb N}}
\renewcommand{\P}{{\mathbb P}}
\newcommand{\Q}{{\mathbb Q}}
\newcommand{\R}{{\mathbb R}}
\newcommand{\T}{{\mathbb T}}
\newcommand{\V}{{\mathbb V}}
\newcommand{\W}{{\mathbb W}}
\newcommand{\X}{{\mathbb X}}
\newcommand{\Y}{{\mathbb Y}}
\newcommand{\Z}{{\mathbb Z}}
\newcommand{\Nwt}{{\rm Nwt}}
\newcommand{\Hdg}{{\rm Hdg}}
\newcommand{\ind}{{\rm ind \,}}
\newcommand{\Br}{{\rm Br}}
\newcommand{\inv}{{\rm inv}}
\newcommand{\Nm}{{\rm Nm}}
\newcommand{\Griff}{{\rm Griff}}
\newcommand{\Image}{\rm Im \,}
\newcommand{\Ev}{\rm Ev \,}
\title[Hypersurfaces]{Unramified cohomology, $\A^1$-connectedness, and Chevalley-Warning problem in Grothendieck Ring}
\author{{Nguyen Le Dang Thi}}
\address{Mathematik, Universit\"at Duisburg-Essen, 45117 Essen, Germany}
\thanks{ This work has been supported by SFB/TR45 "Periods, moduli spaces and arithmetic of algebraic varieties"}
\email{le.nguyen@uni-due.de}
\date{28. 02. 2012}          
\begin{abstract}
We study the Chevalley-Warning problem in the Grothendieck ring $K_0(Var/k)$. We show that the $\A^1$-homotopy theory yields well defined invariants on $K_0(Var/k)/\bold{L}$, in particular the Brauer group is such an invariant. We use this to give a concrete counter-example to the Chevalley-Warning conjecture over a $C_1$-field \cite{BS11}. This also gives a negative answer to the question in \cite[Ques. 3.8]{Bil11}. \\ \\    
\selectlanguage{francais}
\noindent
\textbf{R\'esum\'e:} \textbf{Cohomologie non ramifi\'ee, $\A^1$-connexit\'e et le probl\`eme de Chevalley-Warning dans l'anneau de Grothendieck.} Nous \'etudions le probl\`eme de Chevalley-Warning dans l'anneau de Grothendieck
$K_0 (Var/k) $. Nous montrons que la th\'eorie $\A^1$-homotopie fournit des invariants sur $K_0(Var/k) /\bold{L}$. En particulier le groupe de Brauer est un tel invariant. Nous utilisons cela pour donner un contre-exemple concret \`a la
conjecture de Chevalley-Warning sur un corps $C_1$ \cite{BS11}. Cela donne aussi une r\'eponse n\'egative \`a la question dans \cite[Ques. 3.8]{Bil11}.
\end{abstract}
\maketitle
\section{Introduction}
Let $k$ be a field and $Var/k$ be the category of varieties over $k$. We denote by $K_0(Var/k)$ the Grothendieck ring of varieties over $k$. Over a finite field $k=\F_q$, the Chevalley-Warning theorem (cf. \cite{Ax64}) states that a projective hypersurface $X \subset \P^n$ of degree $d \leq n$ satisfies the congruence formula 
\begin{equation}\label{eqcongruence}
|X(\F_q)| \equiv 1 \mod q.  
\end{equation}
The counting point $X \mapsto |X(\F_q)|$ gives rise to a ring homomorphism 
$$|-|:  K_0(Var/\F_q) \rightarrow \Z, $$
from which one may reformulate the congruence formula \ref{eqcongruence} as $|[X]| \equiv 1 \mod |\bold{L}|$, where we denote by $\bold{L}$ the class of the affine line $[\A^1]$ in $K_0(Var/\F_q)$. The geometric Chevalley-Warning problem for smooth projective hypersurfaces concerns with the following question: 
\begin{ques}\label{quesCW}
Let $k$ be a field and $X \subset \P^n$ be a smooth hypersurface of degree $\leq n$ such that $X(k) \neq \emptyset$. Whether is it true that $[X] \equiv \bold{1} \mod \bold{L}$ in $K_0(Var/k)$, where $\bold{1} = [\Spec k]$?  
\end{ques}
In \cite[3.3]{BS11}, F. Brown and O. Schnetz conjectured that the question \ref{quesCW} is always true for $C_1$-fields. The question \ref{quesCW} over an arbitrary field $k$ is due to H. Esnault in general for the relationship between rational points and the Grothendieck ring $K_0(Var/k)$ (cf. \cite[Ques. 3.7]{Bil11}). In \cite{Li11} the question \ref{quesCW} is formulated over algebraically closed fields of characteristic $0$ under the name geometric Chevalley-Warning conjecture. Some cases, where the question \ref{quesCW} has an affirmative answer for singular hypersurfaces, were worked out in \cite{Bil11} and \cite{Li11}. Using Brauer group, which yields a well defined invariant on $K_0(Var/k)/\bold{L}$, we give a counter-example to the  conjecture of Brown and Schnetz over non-algebraically closed $C_1$-fields.
\begin{thm}\label{thm1}
Let $X$ be a smooth projective geometrically integral variety over a field $k$ of characteristic $0$. If $[X] \equiv \bold{1} \mod \bold{L}$, then $Br(X) \cong Br(k)$. 
\end{thm}
The proof of \ref{thm1} is simple. By Koll\'ar-Larsen-Lunts theorem (cf. \cite{Ko05}, \cite{LL03}), one has $[X] \equiv \bold{1} \mod \bold{L}$ iff $X$ is stably $k$-rational. The fact that the Brauer group $Br(X)$ is a birational invariant is due to Grothendieck \cite[Cor. 7.3, p. 138]{Gro68}. Moreover, one has $Br( \P^n_X) \cong Br(X)$, because $Br(X)$ can be identified with the unramified Brauer group $Br_{nr}(k(X))$ from the exact sequence (cf. \cite[(3.9)]{CT95}) 
$$0 \rightarrow Br(X) \rightarrow Br(k(X)) \rightarrow \bigoplus_{x \in X^{(1)}}H^1_{\acute{e}t}(\kappa(x),\Q/\Z) $$
and the later group $Br_{nr}(k(X))$ gives us a stably birational invariance \cite{CTO89}.  So the theorem follows, since $Br(\P^n_k) \cong Br(k)$. In fact, theorem \ref{thm1} is a special case of a more general invariant coming from strictly $\A^1$-invariant sheaves (see Theorem \ref{thma} below). However, it is enough to produce a counter-example to the geometric Chevalley-Warning conjecture over non-algebraically closed $C_1$-fields. 
\begin{cor}\label{cor2}
Let $k$ be a non-algebraically closed field of $char(k) \neq 3$ and assume $k^{\times} \setminus (k^{\times})^3$ is not empty.  Let $X$ be a smooth cubic surface given by the equation
$$x_0^3+x_1^3+x_2^3+ax_3^3 = 0, $$
where $a \notin (k^{\times})^3$. Then $Br(X)/Br(k)$ is non-trivial. In particular, if $k$ is a non-algebraically closed $C_1$-field of characteristic $0$ with $k^{\times} \setminus (k^{\times})^3 \neq \emptyset$, then $[X]$ is not $\equiv \bold{1} \mod \bold{L}$.
\end{cor}
\begin{proof}
Obviously $X(k) \neq \emptyset$. If $k$ is a non-algebraically closed field with $char(k) \neq 3$ containing a primitive cubic root of unity, then for the smooth cubic surface as above one has $Br(X)/Br(k) = \Z/3 \oplus \Z/3$ (cf. \cite[Ex. 45.3]{Man86} for number fields and \cite[2.5.1]{CTS87} in general). If $k$ has no  primitive cubic roots of unity, the quotient $Br(X)/Br(k)$ is still non-trivial and it is described in \cite[Prop. 2.1]{CTW11}. This gives a negative answer to the question \ref{quesCW} as desired.
\end{proof}
Now let $k$ be an arbitrary field and let $\bold{Ho}_{\A^1}(k)$ be the $\A^1$-homotopy category constructed in \cite{MV01}. For a space $\sX \in \Delta^{op}Sh_{Nis}(Sm/k)$ let $\pi_0^{\A^1}(\sX)$ be the sheaf associated to the presheaf 
$$U \mapsto [U,\sX]_{\A^1} \stackrel{def}{=} \Hom_{\bold{Ho}_{\A^1}(k)}(U,\sX), $$
for $U \in Sm/k$. We say $\sX$ is $\A^1$-connected, if the canonical map $\sX \rightarrow \Spec k$ induces an isomorphism of sheaves $\pi_0^{\A^1}(\sX) \stackrel{\simeq}{\rightarrow} \pi_0^{\A^1}(\Spec k) =  \Spec k$, \cite{AM11}. Let $D_{\A^1}(k)$ denote the $\A^1$-derived category introduced by F. Morel (see e.g \cite[\S 5.2]{Mor12}). Let us denote by $\sA b_k^{\A^1}$ the category of strictly $\A^1$-invariant sheaves (cf. \cite[Def. 7, page 8]{Mor12} or \cite[Def. 4.3.1]{AM11}), it is known that $D_{\A^1}(k)$ has a homological $t$-structure and one can identify $\sA b_k^{\A^1}$ with the heart of this $t$-structure \cite[Lem. 6.2.11]{Mor05}. Thus $\sA b_k^{\A^1}$ is an abelian category by \cite[Thm. 1.3.6]{BBD82}. For a strictly $\A^1$-invariant sheaf $M$ and an irreducible smooth $k$-scheme $X$ we write $M^{nr}(X)$ for the group of unramified elements (\cite[Def. 4.1]{A11}). Now in the context of $\A^1$-derived category one can prove 
\begin{thm}\label{thma}
Let $k$ be a field of characteristic $0$. If $X,Y$ are two irreducible smooth projective $k$-varieties, such that $[X] = [Y]$ in $K_0(Var/k)/\bold{L}$, then $M(X) \cong M(Y)$ for any strictly $\A^1$-invariant sheaf $M \in \sA b_k^{\A^1}$, i.e. $M$ yields a well-defined invariant on $K_0(Var/k)/\bold{L}$. In particular, if $X$ is an integral smooth projective $k$-variety, whose class in $K_0(Var/k)$ satisfies $[X] \equiv \bold{1} \mod \bold{L}$, then $X$ is $\A^1$-connected, hence for any strictly $\A^1$-invariant sheaf $M \in \sA b_k^{\A^1}$ the canonical map $M(k) \rightarrow M^{nr}(X)$ is then a bijection, where $M^{nr}(X)$ denotes the group of unramified elements.
\end{thm}
\begin{rem}{\rm Theorem \ref{thma} is just a simple application of \cite[Thm. 3.9]{A11}. Our example \ref{cor2} shows that this smooth cubic surface is $\A^1$-disconnected over non-algebraically closed fields, while \cite[Cor. 2.4.7]{AM11} asserts  that a smooth proper surface over an algebraically closed field of characteristic $0$ is $\A^1$-connected if and only if it is rational.  
}
\end{rem}
\section{Proof of \ref{thma}}
By Koll\'ar-Larsen-Lunts theorem (cf. \cite{Ko05}, \cite{LL03}), one has an isomorphism  
$$K_0(Var/k)/\bold{L} \rightarrow \Z[SB], $$
where the right hand side denotes the free abelian group generated over the set of stably birational equivalences of smooth projective varieties. So if $[X] = [Y]$ in $K_0(Var/k)/\bold{L}$, then $X$ is stably $k$-birational to $Y$. We have then $\bold{H}_0^{\A^1}(X) \cong \bold{H}_0^{\A^1}(Y)$ by \cite[Thm. 3.9]{A11}. By representing theorem \cite[Lem. 3.3]{A11}, which asserts that  
$$H^0_{Nis}(X,M) = \Hom_{\sA b^{\A^1}_k}(\bold{H}_0^{\A^1}(X),M),$$
one obtains $M(X) \cong M(Y)$. Remark that one has $M(X) = M^{nr}(X)$, if $X$ is an irreducible smooth $k$-scheme (\cite[Lem. 4.2]{A11}).  Now if $X$ is an integral smooth projective $k$-variety with $[X] \equiv \bold{1} \mod \bold{L}$ in $K_0(Var/k)$, then $X$ is stably $k$-rational. From \cite[Prop. 1.4]{CTS07} one knows that $X$ is then retract $k$-rational in sense of Saltman. By \cite[Thm. 2.3.6]{AM11} $X$ is $\A^1$-chain connected, hence $\A^1$-connected by \cite[Lem. 6.1.3]{Mor05}. Thus the theorem is proved and we see also immediately that \ref{thm1} is a special case of \ref{thma} by \cite[Prop. 4.3.8]{AM11}. 
\\ \\
\thanks{$\bold{Acknowledegements :}$ I wish to thank J.-L. Colliot-Th\'el\`ene, H. Esnault, M. Levine and O. Wittenberg for many help and support during the writing of this note. Especially my thank goes to my advisor H. Esnault for her patience to correct many errors in this note. }  
\bibliographystyle{plain}
\renewcommand\refname{References}

\end{document}